\documentclass[a4paper,11pt]{article}
\usepackage[utf8]{inputenc}
\usepackage{amsmath}
\usepackage{amsthm}
\usepackage{amssymb}
\usepackage{hyperref}
\usepackage{xcolor}

\hypersetup{
    colorlinks,
    linkcolor={blue!80!black},
    citecolor={blue!80!black},
    urlcolor={blue!80!black}
}

\newtheorem{definition}{Definition}
\newtheorem{proposition}{Proposition}
\newtheorem{lemma}{Lemma}
\newtheorem{corollary}{Corollary}
\usepackage[T1]{fontenc}

\newcommand{\Liep}{\Lie'}
\newcommand{\Cinf}{C^\infty}
\newcommand{\ud}{\mathrm{d}}
\newcommand{\pd}{\partial}
\newcommand{\e}{\varepsilon}
\newcommand{\bN}{\mathbb{N}}
\newcommand{\bR}{\mathbb{R}}
\newcommand{\cA}{\mathcal{A}}
\newcommand{\ccD}{\mathcal{D}}

\newcommand{\cG}{\mathcal{G}}

\newcommand{\csub}{\subset \subset}
\newcommand{\coleq}{\mathrel{\mathop:}=}
\newcommand{\fX}{\mathfrak{X}}
\providecommand{\norm}[1]{\left\lVert#1\right\rVert}
\DeclareMathOperator{\Fl}{Fl}
\DeclareMathOperator{\Div}{div}
\DeclareMathOperator{\dist}{dist}

\DeclareMathOperator{\supp}{supp}

\newcommand{\LX}{\Lie_X}
\newcommand{\Lie}{\mathrm{L}}

\newcommand{\sk}[2]{\widetilde\cA_{#1}(#2)}
\newcommand{\lsk}[2]{\widetilde\cA_{#1}(#2)}
\newcommand{\ub}[2]{\hat\cA_{#1}(#2)}
\newcommand{\lub}[2]{\cA_{#1}(#2)}
\providecommand{\abso}[1]{\left\lvert#1\right\rvert}

\begin{document}

\title{Approximation properties of\\local smoothing kernels}
\date{}
\author{Eduard A.~Nigsch\thanks{Wolfgang-Pauli-Institut, Oskar-Morgenstern-Platz 1, 1090 Vienna, Austria. email: eduard.nigsch@univie.ac.at}}
\maketitle

\begin{abstract}
We study some properties of smoothing kernels and their local expression as they appear in the construction of Colombeau-type generalized function algebras which are diffeomorphism invariant.
\end{abstract}

\section{Introduction}

Smoothing kernels as introduced in \cite[Section 3]{global} are key elements in the construction of diffeomorphism invariant Colombeau algebras, which are generalized function algebras featuring an embedding of smooth functions and distributions while retaining as much of the classical properties and operations as possible under the constraint of the Schwartz impossibility result \cite{Schwartz}.

At the heart of such a construction (illustrated in the scalar case) lies the regularization of a distribution by representing it as a family of smooth functions, the classical example being convolution $u \star \psi$ of a distribution  $u \in \ccD'(\bR^n)$ with a test function $\psi \in \ccD(\bR^n)$. Varying $\psi$ sufficiently one captures all information about $u$: given $\psi$ with integral 1 and setting $\psi_\e(x) = \e^{-n} \psi(x/\e)$ for $\e \in I \coleq (0,1]$, $u \star \psi_\e$ converges to $u$ distributionally for $\e \to 0$. Consequently, distributions are embedded into generalized function algebras by convoluting them with suitable mollifiers.

In the setting of diffeomorphism invariant Colombeau algebras this convolution has to be done in a more general way: generalized functions there are mappings $R(\varphi, x)$ depending on mollifiers $\varphi$ and points $x$, the embedded image of a distribution $t$ is simply given by $R(\varphi,x) \coleq \langle t, \varphi(.-x) \rangle$ \cite[Definition 5.5]{found}. In order to obtain an algebra one then has to perform a quotient construction (which Colombeau algebras are distinctive for) by considering the asymptotics of $R(S_\e\phi(\e,x), x)$ \cite[Chapter 7.1]{found}, where $\phi$ is a mollifier smoothly indexed by $(\e,x) \in I \times \bR^n$ satisfying certain properties.
For Colombeau algebras on manifolds one furthermore has to change to a coordinate-free but equivalent formalism that does not rely on the linear structure of $\bR^n$, which directly results in the definition of smoothing kernels (cf. \cite[Section 3]{global} or Definition \ref{smoothkern} below) and the global algebra $\hat\cG(M)$ of \cite{global}. Technically, this leads to expressions of the form $\int f(y) \tilde \phi(\e,x)(y)\, \ud y$ (where we set $\tilde \phi(\e,x) \coleq \e^{-n}\phi(\e,x)((y-x)/\e)$) which shall be the main focus of this article.

Note that the seemingly harmless relation between $\phi$ and $\tilde \phi$ gives rise to two different (but equivalent) formulations of the local theory of $\cG^d(\Omega)$, the diffeomorphism invariant Colombeau algebra on an open subset $\Omega \subseteq \bR^n$ \cite[Section 5]{found}. Due to the fact that the local theory is built on $\phi$-like test objects, the global formulation of $\hat\cG(M)$ on a manifold \cite{global} invariably must involve a change of formalism whenever local calculations are performed, which is an unnecessary technical complication.

We will therefore introduce local smoothing kernels as immediate equivalents of their global version, which serves to decouple the global theory of \cite{global} from the local diffeomorphism invariant theory of $\cG^d(\Omega)$ and thus makes an independent, clear-cut formulation possible. Furthermore the theory gets clearer and simpler, e.g., the proof of diffeomorphism invariance and results related to association (as in \cite[Section 9]{global2}) come more easily with the proper formalism. For this purpose we will study approximation properties of smoothing kernels, which is their most needed feature.

\section{Preliminaries}

We will use the following notions in this article: $M,N$ are orientable paracompact smooth Hausdorff manifolds of finite dimension $n$, diffeomorphisms will be assumed to be orientation preserving, accordingly.
$\fX(M)$ is the space of smooth vector fields on $M$, $\LX$ is the usual Lie derivative. 
$\Omega$ is a subset of $\bR^n$,
$U \csub V$ means that $U$ is a compact subset of the interior of $V$. $C^\infty$ denotes spaces of smooth mappings, we set $\hat\cA_0(M) \coleq \{ \omega \in \Omega^n_c(M)\ |\ \int\omega = 1 \}$ and $\lub0\Omega \coleq \{\varphi \in \ccD(\Omega)\ |\ \int \varphi = 1 \}$, where $\ccD(\Omega)$ is the space of test functions on $\Omega$. $\Omega^n_c(M)$ resp.\ $\Omega^n_c(\Omega)$ are $n$-forms on $M$ resp.\ $\Omega$.
$O(\e^m)$ means the usual Landau symbol, always for $\e \rightarrow 0$.
As to the notion of smoothness employed here, we refer the reader to \cite{KM, Froe} where calculus on convenient vector spaces is presented.
For a deeper understanding and some background material we point to the related variants of Colombeau algebras which exist in the literature \cite{GKOS, found, global, global2}.

\section{Smoothing Kernels}

Smoothing kernels basically are $n$-forms depending on $\e \in I$ and an additional space variable satisfying certain properties needed for the construction of Colombeau algebras. As a preliminary we will define two-point $n$-forms on a manifold, their Lie derivative in both variables, and pullback. We will only be concerned with compactly supported $n$-forms throughout. Consider functions $\Phi \in \Cinf(M, \Omega^n_c(M))$, which shall be called \emph{two-point $n$-forms}. All subsequent results remain valid if $\Phi$ additionally depends on $\e \in I$, as it will for smoothing kernels.

\begin{definition}\label{tp_lie}On $\Cinf(M, \Omega^n_c(M))$ we define two Lie derivatives, $\LX\Phi \coleq \LX \circ \Phi$ and $(\Liep_X\Phi)(p) \coleq \left.\frac{\ud}{\ud t}\right|_{t=0} \Phi(\Fl^X_t p)$ for $p \in M$.
\end{definition}

\begin{proposition}$\LX$ and $\Liep_X$ are smooth mappings into $\Cinf(M, \Omega^n_c(M))$.
\end{proposition}
\begin{proof}
By \cite[Theorem 4.1]{found} for smoothness of $\LX$ on $\Omega^n_c(M)$ it suffices to check boundedness of $\LX\colon \Omega^n_{c,K}(M) \to \Omega^n_c(M)$ for compact sets $K \csub M$, which follows directly from \cite[Proposition A.2 (2) (i)]{global2}. For $\Liep_X$, $\Phi \in C^\infty(M, \Omega^n_c(M))$ if and only if for each chart $(U, \varphi)$ of an atlas $\Phi \circ \varphi^{-1}$ is in $C^\infty(\varphi(U), \Omega^n_c(M))$. Denote by $\alpha(t,x)$ the local flow of $X$ in the chart. Then for fixed $p$, $\alpha(t,\varphi(p))$ exists for $t$ in a neighborhood of $0$ and
$\left.\frac{\ud}{\ud t}\right|_{t=0} \Phi(\Fl^X_t p) = \left.\frac{\ud}{\ud t}\right|_{t=0} (\Phi \circ \varphi^{-1})(\varphi \circ \Fl^X_t p)
= \left.\frac{\ud}{\ud t}\right|_{t=0} (\Phi \circ \varphi^{-1})(\alpha(t,\varphi(p)))
= \ud(\Phi \circ \varphi^{-1})(\varphi(p)) \cdot X(\varphi(p))$
(where the local expression of $X$ is denoted by the same letter). From this we see that the limit exists and is smooth.
\end{proof}

\begin{definition}\label{tp_pullback}Given a diffeomorphism $\mu\colon M \to N$ and a two-point $n$-form $\Phi \in \Cinf(N, \Omega^n_c(N))$, the \emph{pullback} $\mu^* \Phi \in C^\infty(M, \Omega^n_c(M))$ of $\Phi$  is defined as $\mu^* \Phi \coleq \mu^* \circ \Phi \circ \mu$. 
\end{definition}

Now we will examine how these Lie derivatives translate under pullbacks.

\begin{lemma}\label{tpliediffeo}Let $\Phi \in C^\infty(N, \Omega^n_c(N))$ be a two-point $n$-form and $\mu\colon M \to N$ a diffeomorphism. Then $\LX (\mu^* \Phi) = \mu^* ( \Lie_{\mu_*X}\Phi)$ and $\Liep_X (\mu^* \Phi) = \mu^* ( \Lie_{\mu_*X} \Phi)$.
\end{lemma}
\begin{proof}
First, we have
$ \LX(\mu^* \Phi) = \LX \circ \mu^* \circ \Phi \circ \mu = \mu^* \circ \Lie_{\mu_*X} \circ \Phi \circ \mu = \mu^* ( \Lie_{\mu_*X} \Phi)$. Second, because $\mu^*\colon \Omega^n_c(M) \to \Omega^n_c(M)$ is linear and continuous we have
\begin{align*}
\Liep_X (\mu^*\Phi)(p)  & = \left.\frac{\ud}{\ud t}\right|_{t=0} (\mu^*\Phi)(\Fl^X_t p) = \left.\frac{\ud}{\ud t}\right|_{t=0} (\mu^* \circ \Phi)((\mu \circ \Fl^X_t) (p) ) \\
&= \left.\frac{\ud}{\ud t}\right|_{t=0} (\mu^* \circ \Phi)((\Fl^{\mu_* X}_t \circ \mu)(p)) = \mu^* \left( \left.\frac{\ud}{\ud t}\right|_{t=0} \Phi((\Fl^{\mu_* X}_t \circ \mu)(p)) \right) \\
&= \mu^*((\Lie_{\mu_*X}\Phi)(\mu(p))) =\mu^* ( \Lie_{\mu_*X} \Phi).\qedhere
\end{align*}
\end{proof}

With $\e$-dependence added we can give the definition of (global) smoothing kernels (\cite[Definition 3.3]{global}).

\begin{definition}\label{smoothkern}
\begin{enumerate}
\item A mapping $\Phi \in C^\infty(I \times M, \ub0M)$ is called a \emph{smoothing kernel} if it satisfies the following conditions w.r.t.\ any Riemannian metric $g$ on $M$:
\begin{enumerate}
\item[(i)] $\forall K \csub M$ $\exists \e_0>0, C>0$ $\forall p \in K$ $\forall \e\le\e_0$: $\supp \Phi(\e,p) \subseteq B^g_{\e C}(p)$ (the open metric ball of radius $\e C$ at $p$).
\item[(ii)] $\forall K \csub M$ $\forall l,m \in \bN_0$ $\forall \zeta_1,\dotsc,\zeta_l,\theta_1,\dotsc,\theta_m \in \fX(M)$ the norm of $(\Lie_{\theta_1}\dotsc \Lie_{\theta_m} (\Liep_{\zeta_1} + \Lie_{\zeta_1}) \dotsc (\Liep_{\zeta_l} + \Lie_{\zeta_l})\Phi)(\e,p)(q)$ is $O(\e^{-n-m})$ uniformly for $p \in K$ and $q \in M$.
\end{enumerate}
The space of all smoothing kernels is denoted by $\sk 0M$.
\item For each $k \in \bN$ denote by $\sk kM$ the set of all $\Phi \in \sk 0M$ such that $\forall f \in C^\infty(M)$ and $K \csub M$, $\abso{f(p) - \int_M f \cdot \Phi(\e,p)} = O(\e^{k+1})$ uniformly for $p \in K$.
\end{enumerate}
\end{definition}

Note that elements of $\sk 0M$ satisfy the asymptotics of (2) for $k=0$. One even has $\sup_{x \in K} \abso{f(p,p) - \int_M f(p,.)\cdot \tilde\phi(\e,p) } = O(\e^{k+1})$ \cite[Lemma 3.6]{global2}. That this definition indeed is independent of the metric used follows from the next lemma, which is obtained similarly to \cite[Lemma 3.4]{global}.

\begin{lemma}\label{lama}Let $(M, g)$ and $(N, h)$ be Riemannian manifolds and $\mu\colon M \to N$ a diffeomorphism. Then $\forall K \csub M$ $\exists C>0$ such that
 \begin{enumerate}
  \item[(i)] $\norm{(\mu^*\omega)(p)}_g \le C \norm{\omega(\mu(p))}_h$ $\forall \omega \in \Omega^n_c(N)$ $\forall p \in K$.
  \item[(ii)] $B_r^g(p) \subseteq \mu^{-1}(B_{r C}^h(\mu(p))) = B_{rC}^{\mu^*h}(p)$ $\forall r>0$ $\forall p \in K$.
 \end{enumerate}
\end{lemma}

Now we define a suitable action of diffeomorphisms on smoothing kernels and show that this leaves them invariant.

\begin{proposition}\label{lahmah}Given $\Phi \in \sk kN$ and a diffeomorphism $\mu\colon M \to N$ the map $\Psi\colon I \times M \to \Omega^n_c(M)$ defined by $\Psi(\e,p) \coleq \mu^*(\Phi(\e,\mu(p)))$ is in $\sk kM$.
\end{proposition}
\begin{proof}Except for the $\e$-variable this is the pullback of Definition \ref{tp_pullback}, so it maps into the right space. For (i) of Definition \ref{smoothkern}, let $K \csub M$ be given. Given any Riemannian metric $h$ on $N$ there are constants $\e_0>0$ and $C>0$ such that for all $p \in K$ and $\e \le \e_0$ the support of $\mu^*(\Phi(\e, \mu (p)))$ is contained in $\mu^{-1}(B^h_{\e C}(\mu (p)))$. By Lemma \ref{lama} (ii) for any Riemannian metric $g$ on $M$ there is a constant $L>0$ such that $\mu^{-1}(B^h_{\e C}(\mu (p))) \subseteq B^g_{\e LC}(p)$ for all $p \in K$ and $\e \le \e_0$.

For (ii), given any vector fields $\theta_1, \dotsc, \theta_m, \zeta_1,\dotsc,\zeta_l \in \fX(M)$ we see that
\[ \Lie_{\theta_1} \dotsc \Lie_{\theta_m} ( \Lie_{\zeta_1} + \Liep_{\zeta_1}) \dotsc (\Lie_{\zeta_l} + \Liep_{\zeta_l})\Psi \]
equals (by Lemma \ref{tpliediffeo})
\[ \mu^*(\Lie_{\mu_*\theta_1} \dotsc \Lie_{\mu_*\theta_m} (\Lie_{\mu_*\zeta_1} + \Liep_{\mu_*\zeta_1}) \dotsc (\Lie_{\mu_*\zeta_l} + \Liep_{\mu_*\zeta_l})\Phi) \]
whence by Lemma \ref{lama} (i) the assertion on the derivatives of $\Psi$ follows from the defining properties of $\Phi$. Finally, the approximation property follows directly from writing down the corresponding integral.
\end{proof}

\begin{definition}Given $\Phi \in \sk kN$ and a diffeomorphism $\mu\colon M \to N$, the map $\mu^*\Phi\colon (\e,p) \mapsto \mu^*(\Phi(\e,\mu(p))$ in $\sk kM$ is called the pullback of $\Phi$.
\end{definition}

\section{Local smoothing kernels}

In this section we will introduce local versions of the spaces of smoothing kernels $\sk kM$ and examine their properties in some detail.

Locally $n$-forms can be identified with smooth functions, which is made precise by the vector space isomorphism $\hat\lambda\colon \Omega^n_c(\Omega) \to \ccD(\Omega)$ assigning to $\omega \in \Omega^n_c(\Omega)$ the function $x \mapsto \omega(x)(e_1,\dotsc,e_n)$ in $\ccD(\Omega)$, where $e_1,\dotsc,e_n$ is the standard basis of $\bR^n$; its inverse is the mapping $f \mapsto f \, \ud x^1 \wedge \dotsc \wedge \ud x^n$ and both assignments are continuous. The local equivalent of the Lie derivative is simply the directional derivative: the definition $(\LX f)(x) \coleq (\ud f)(x) \cdot X(x) + \Div X(x) \cdot f(x)$ for $f \in \ccD(\Omega)$ and $x \in \Omega$ gives a smooth map $\LX\colon \ccD(\Omega) \to \ccD(\Omega)$. It also satisfies the relation $\LX(\mu^*f) = (\Lie_{\mu_*X}f) \circ \mu$ with $(\mu_*X)(x) \coleq (\ud \mu)(\mu^{-1}x) \cdot X(x)$. For any $n$-form $\omega \in \Omega^n_c(\Omega)$ we immediately obtain the properties $\omega = \hat\lambda(\omega) \, \ud x^1 \wedge \dotsc \wedge \ud x^n$, $\int \hat\lambda(\omega)(x)\, \ud x = \int \omega$, $\supp \hat\lambda(\omega) = \supp \omega$, and $\hat\lambda(\LX \omega) = \LX (\hat\lambda(\omega))$.

We have a bornological vector space isomorphism
\[ C^\infty(\Omega, \Omega^n_c(\Omega)) \cong \Cinf (\Omega, \ccD(\Omega)) \]
realized by the mapping $\hat\lambda_*\colon \tilde \phi \mapsto \hat\lambda \circ \tilde \phi$.

\begin{definition}On $C^\infty(\Omega, \ccD(\Omega))$ we define two Lie derivatives: $\LX \tilde \phi \coleq \LX \circ \tilde \phi$ and $(\Liep_X \tilde \phi) (x) \coleq (\ud \tilde \phi)(x) \cdot X(x)$ for $x \in \Omega$ and $X \in \Cinf(\Omega, \bR^n)$.
\end{definition}

From $\LX ( \hat\lambda_* \tilde \phi) = \LX \circ \hat\lambda \circ \tilde \phi = \hat\lambda \circ \LX \circ \tilde \phi$ and $\Liep_X ( \hat\lambda_* \tilde \phi)(x) = \ud (\hat\lambda \circ \tilde \phi)(x) \cdot X(x) = \hat\lambda((\ud \tilde \phi)(x) \cdot X(x)) = \hat\lambda ((\Liep_X \tilde \phi)(x))$ we see that $\hat\lambda_*$ commutes with both $\LX$ and $\Liep_X$.
As above we have the following.
\begin{lemma}\label{nformtestfct}
 $\hat\lambda_*\colon C^\infty(I \times \Omega, \Omega^n_c(\Omega)) \to C^\infty(I \times \Omega, \ccD(\Omega))$ is a bornological vector space isomorphism with inverse $(\hat\lambda^{-1})_*$ and commutes with $\LX$ and $\Liep_X$.
\end{lemma}

We will see in Proposition \ref{localprop} that the defining properties of a smoothing kernel $\Phi \in C^\infty(I \times U, \ub0U)$ translate verbatim to its \emph{local expression}, which is defined as $\tilde \phi \coleq \hat\lambda_* ( \varphi_* \Phi) \in C^\infty(I \times \varphi(U), \lub0{\varphi(U)})$. We thus define local smoothing kernels as follows.

\begin{definition}\label{localsmoothkern}\begin{enumerate}
\item[(1)]
A mapping $\tilde \phi \in C^\infty(I \times \Omega, \lub0\Omega)$ is called a \emph{local smoothing kernel} (on $\Omega)$ if it satisfies the following conditions:
\begin{enumerate}
\item[(i)] $\forall K \csub \Omega$ $\exists \e_0>0, C>0$ $\forall x \in K$ $\forall \e\le \e_0$: $\supp \tilde\phi(\e,x) \subseteq B_{\e C}(x) \subseteq \Omega$, where $B_{\e C}(x)$ is the Euclidean ball.
\item[(ii)] $\forall K \csub \Omega$ $\forall \alpha,\beta \in \bN_0^n$ we have the estimate $\abso{ (\pd_y^\beta \pd_{x+y}^\alpha \tilde \phi) (\e,x)(y)} = O(\e^{-n-\abso{\beta}})$ uniformly for $x \in K$ and $y \in \Omega$.
\end{enumerate}
The space of all local smoothing kernels is denoted by $\lsk 0\Omega$.
\item[(2)] For each $k \in \bN$ denote by $\lsk k\Omega$ the set of all $\tilde\phi \in \lsk 0\Omega$ such that for all $f \in C^\infty(\Omega)$ and $K \csub \Omega$, $\abso{f(x) - \int_\Omega f(y)\tilde\phi(\e,x)(y)\, \ud y} = O(\e^{k+1})$ uniformly for $x \in K$.
\end{enumerate}
\end{definition}

Again, elements of $\lsk 0\Omega$ satisfy the asymptotics of (2) for $k=0$ and by the usual methods (Taylor expansion of $f$) one even has
\[ \abso{f(x,x) - \int_\Omega f(x,y)\tilde\phi(\e,x)(y)\, \ud y } = O(\e^{k+1}) \] uniformly for $x \in K$.

\begin{proposition}\label{localprop}For a chart $(U, \varphi)$ there is a bornological vector space isomorphism $\sk kU \cong \lsk k{\varphi(U)}$ given by $\Phi \mapsto \hat\lambda_* ( \varphi_* \Phi)$.
\end{proposition}
\begin{proof}
 Set $\tilde \phi \coleq \hat\lambda_* ( \varphi_* \Phi)$. For (i) of Definition \ref{localsmoothkern} fix $K \csub \Omega$, then there are $\e_0>0$ and $C>0$ such that $\supp \Phi(\e,p) \subseteq B_{\e C}(p)$ for all $p \in \varphi^{-1}(K)$ and $\e \le \e_0$. We may assume that $\e_0 C < \dist(\varphi^{-1}(K), \pd U)$. Then $
 \supp \tilde \phi(\e, x) = \supp\, ( \varphi_* \Phi)(\e,x) = \supp \varphi_*(\Phi(\e,\varphi^{-1}(x)))
\subseteq \varphi ( B_{\e C}(\varphi^{-1}(x))) \subseteq B_{\e C'}(x)$ for some $C'>0$ by Lemma \ref{lama} (ii). (ii) of Definition \ref{localsmoothkern} is a consequence of Lemmata \ref{nformtestfct}, \ref{tpliediffeo} and \ref{lama} (i).
\end{proof}

\section{Approximation using smoothing kernels}

The practical importance of smoothing kernels lies in their approximation properties; in the context of Colombeau algebras one often has to consider expressions like $\int f(y)\tilde \phi(\e,x)(y)\, \ud y$. In particular this appears in the proof of injectivity of the embedding of distributions and in questions of association (cf. \cite[Section 9]{global2}). Variants of the integral just mentioned involve taking derivatives and integrating over $x$ instead of $y$.

The behaviour of these integrals can be guessed by considering the simple example $\tilde \phi(\e,x) \coleq \e^{-n} \varphi((.-x)/\e^n)$ for some $\varphi \in \cA_0(\Omega)$. In this case we have (by Taylor expansion of $f$, partial integration, and the fact that $\pd_{x+y}\tilde \phi = 0$):
\begin{enumerate}
 \item[(a)] $\int f(x,y) \tilde\phi (\e,x,y) \, \ud y \to f(x,x)$,
 \item[(b)] $\int f(x,y) \tilde\phi (\e,x,y) \, \ud x \to f(y,y)$,
 \item[(c)] $\int f(x,y) (\pd_{y_i} \tilde\phi) (\e,x,y) \, \ud y \to -(\pd_{y_i} f)(x,x)$,
 \item[(d)] $\int f(x,y) (\pd_{x_i} \tilde\phi) (\e,x,y) \, \ud x \to -(\pd_{x_i} f)(y,y)$. 
 \item[(e)] $\int f(x,y) (\pd_{x_i} \tilde\phi) (\e,x,y) \, \ud y \to (\pd_{y_i} f)(x,x)$,
 \item[(f)] $\int f(x,y) (\pd_{y_i} \tilde\phi) (\e,x,y) \, \ud x \to (\pd_{x_i} f)(y,y)$,
\end{enumerate}
uniformly on compact sets, and analogously for higher derivatives. We will see that the same results hold for arbitrary smoothing kernels (for the integral over $x$ we assume $f$ to have compact support): from (a) and (b) (Remark after Definition \ref{localsmoothkern} and Proposition \ref{ableit}) partial integration gives (c) and (d), (e) and (f) are handled by Corollary \ref{blahfoo}.

For derivatives of $\tilde \phi(\e,x)(y)$ in each slot we will employ, in the obvious sense, the notation $(\pd_x^\alpha \pd_y^\beta \tilde \phi)(\e,x)(y)$ and we set $\pd_{x+y}^\alpha = (\pd_x + \pd_y)^\alpha$. Furthermore, we write $\tilde \phi(\e,x,y)$ instead of $\tilde \phi(\e,x)(y)$ where it is convenient.

\begin{lemma}\label{dersupp}Let $\tilde \phi \in \lsk0\Omega$ be a local smoothing kernel. Then $\forall K \csub \Omega$ $\exists C>0$ $\exists \e_0>0$ such that $\supp \pd_x^\alpha\pd_y^\beta \tilde \phi(\e,x) \subseteq B_{C\e}(x)$ for all $x \in K$, $\e \le \e_0$, and $\alpha,\beta \in \bN_0^n$.
\end{lemma}
\begin{proof}
 As $\pd_y$ preserves the support we set $\beta=0$. Given any $\delta>0$ with $\overline{B}_\delta(K) \subseteq \Omega$ we know that there are $\e_0>0$ and $C>0$ such that $\supp \tilde \phi(\e,x) \subseteq B_{\e C}(x)$ for all $x \in \overline{B}_\delta(K)$ and $\e \le \e_0$. Choose any $C' > C$ and suppose $\alpha=e_{i_1} + \dotsc + e_{i_k}$ with $k=\abso{\alpha}$, then $\pd_x^\alpha \tilde \phi(\e,x)$ is given by derivatives at $t_1,\dotsc,t_k=0$ of $\tilde \phi ( \e, x + t_1 e_{i_1} + \dotsc + t_k e_{i_k} )$; clearly for small $t_i$ the support of the difference quotient is in $B_{\e C}(x + t_1 e_{i_1} + \dotsc + t_k e_{i_k}) \cup B_{\e C}(x) \subseteq B_{\e C'}(x)$.
\end{proof}

Recall that by simple Taylor expansion $\abso{\int_\Omega f(x,y) \tilde \phi(\e,x,y)\,\ud y - f(x,x)} = O(\e)$ and $\abso{\int_\Omega f(x,y) (\pd_{x+y}^\alpha \tilde \phi)(\e,x,y)\,\ud y} = O(\e)$ for $\abso{\alpha}>0$ uniformly for $x$ in compact sets. We will now show the equivalent when the integral is over $x$.

\begin{proposition}\label{ableit} Let $\tilde \phi \in \lsk0\Omega$ be a local smoothing kernel and $f \in C^\infty(\Omega \times \Omega)$ such that there is $K \csub \Omega$ with $\supp f(\cdot, y) \subseteq K$ for all $y \in \Omega$.
Then
\begin{align*}
\abso{\int_\Omega f(x,y) \tilde \phi(\e,x,y)\,\ud x - f(y,y)} &= O(\e)\qquad\textrm{and} \\
\abso{\int_\Omega f(x,y) (\pd_{x+y}^\alpha \tilde \phi)(\e,x,y)\,\ud x} &= O(\e)\qquad(\abso{\alpha}>0)
\end{align*}
uniformly for $y \in \Omega$.
\end{proposition}
\begin{proof}
Without limitation of generality we can assume that there are $r>0$ and $a \in \Omega$ such that $K \subseteq \overline{B}_{r}(a) \subseteq \overline{B}_{4 r(a)} \subseteq \Omega$ for all $y \in \Omega$. In fact, any $K \csub \Omega$ can be written as the sum of finitely many compact sets contained in suitable closed balls which lie in $\Omega$. If the result holds for each of these, it holds for $K$.

The integral then is over $x \in B_{r}(a)$. By Lemma \ref{dersupp} $\exists \e_0>0,C>0$: $\supp\, (\pd_{x+y}^\beta \tilde \phi)(\e,x) \subseteq B_{C\e}(x)$ for all $\beta \in \bN_0^n$, $x \in \overline{B}_{4r}(a)$, and $\e < \e_0$. For $\e < \e_0$ and $x \in B_{r}(a)$ this implies $\supp\, (\pd_{x+y}^\alpha \tilde \phi)(\e,x) \subseteq B_{r+C\e}(a)$ thus we only have to consider $y$ in this set, as for $y \not\in B_{r+C\e}(a)$ the integral vanishes.
We furthermore note that for $\e<\min(\e_0, (r/(4C)))$ and $y \in B_{r+C\e}(a)$ we have $\supp (\pd_{x+y}^\alpha \tilde \phi)(\e,y) \subseteq B_{C\e}(y) = 2y - B_{C\e}(y) \subseteq 2y-B_{r+2C\e}(a) \subseteq B_{3r + 4C\e}(a) \subseteq B_{4r}(a) \subseteq \Omega$ and thus the above integral equals $\int_{B_{r+2C\e}(a)} f(x,y)(\pd_{x+y}^\alpha \tilde \phi)(\e,y,2y-x)\, \ud x$ and we can write it as
\begin{multline}\label{september_alpha}
\int_{B_{r+2C\e}(a)} f(x,y)\bigl((\pd_{x+y}^\alpha \tilde \phi)(\e,x,y) - (\pd_{x+y}^\alpha \tilde \phi)(\e,y,2y-x)\bigr)\, \ud x \\
+ \int_{B_{r+2C\e}(a)} f(x,y)(\pd_{x+y}^\alpha \tilde \phi)(\e,y,2y-x)\, \ud x
\end{multline}
for $y \in B_{r+C\e}(a)$ and $\e \le \min(\e_0, r/(4C))$. For $x \in B_{r+2C\e}(a) \subseteq B_{4r}(a)$ the Taylor formula \cite[Theorem 5.12]{KM} gives (due to $2y-x \in B_{3r +4C\e}(a) \subseteq B_{4r}(a) \subseteq \Omega$)
\begin{multline*}
 (\pd_{x+y}^\alpha \tilde \phi)(\e, x, y) = (\pd_{x+y}^\alpha \tilde \phi)(\e,y,2y-x)\\
+ \int_0^1 \ud (\pd_{x+y}^\alpha \tilde \phi)(\e, y+t(x-y), 2y-x+t(x-y)) \cdot (x-y, x-y)\, \ud t
\end{multline*}
where the differential $\ud$ is with respect to the pair of variables $(x,y)$.
Then the first summand of \eqref{september_alpha} is given by (substituting $x=y+\e z$)
\begin{multline}\label{september_3}
\int_{B_{r/\e+2C}(a-y)} f(y+\e z, y) \cdot\\
 \int_0^1 \ud (\pd_{x+y}^\alpha \tilde \phi)(\e, y+t\e z, y+(t-1)\e z) \cdot (\e z, \e z)\, \ud t\, \e^n\, \ud z
\end{multline}
By linearity of the differential we note that $\ud (\pd_{x+y}^\alpha \tilde \phi)(\e, y+t \e z, y + (t-1) \e z) \cdot (\e z, \e z )$ equals $\e \sum_{i=1}^n (\pd_{x+y}^{\alpha+e_i} \tilde \phi)(\e, y + t \e z , y + (t-1) \e z ) z^i$ where $z=(z^1,\dotsc,z^n)$. From the properties of local smoothing kernels we have that $\abso{ (\pd_{x+y}^{\alpha+e_i}\tilde \phi)(\e,x,y)} = O(\e^{-n})$ uniformly for $x \in \overline{B}_{r+2C\e}(a)$ and $y \in \Omega$ and also $\supp\, (\pd_{x+y}^{\alpha+e_i} \tilde \phi)(\e,x) \subseteq B_{C\e}(x)$ for $x \in \overline{B}_{4r}(a)$ and all $\e < \e_0$. In \eqref{september_3} we only need to integrate over those $z$ such that $y + (t-1)\e z \in B_{C\e}(y + t\e z)$, i.e., $\abso{ y + (t-1)\e z - y - t \e z } = \abso{\e z} < C \e$ which is implied by $\abso{z} < C$. Thus \eqref{september_3} becomes
\begin{multline*}
 \e \cdot \int_{B_{r/\e+2C}(a-y) \cap B_C(0)} f(y+\e z, y) \cdot\\
\int_0^1 \sum_{i=1}^n (\pd_{x+y}^{\alpha+e_i} \tilde \phi)(\e, y + t \e z, y + (t-1) \e z)z^i\, \ud t \, \e^n\, \ud z.
\end{multline*}

By what was said this can be estimated by
$O(\e)$ uniformly for $y \in B_{4r}(a)$ and thus for $y \in \Omega$.
It remains to examine the second term of \eqref{september_alpha}
for $y \in B_{r+C\e}(a)$. With Taylor expansion in the first slot of $f$ this is
\begin{multline*}
f(y,y) \cdot \int_{B_{r+2C\e}(a)} (\pd_{x+y}^\alpha \tilde \phi)(\e,y,2y-x)\,\ud x\\
+ \int_{B_{r+2C\e}(a)} \int_0^1 (\ud_1 f) (y+t(x-y),y) \cdot (x-y)\, \ud t\, (\pd_{x+y}^\alpha \tilde \phi)(\e, y, 2y-x)\,\ud x
\end{multline*}
Substituting $2y-x=z$, because $\supp \tilde \phi(\e,y) \subseteq B_{C\e}(y) = 2y - B_{C\e}(y) \subseteq 2y - B_{r + 2C\e}(a) \subseteq B_{3r + 4C\e}(a) \subseteq B_{4r}(a)$ for $y \in B_{r+C\e}(a)$ the first integral is given by
$\int_{\Omega} (\pd_{x+y}^\alpha \tilde \phi)(\e,y,z)\, \ud z$ which is 1 for $\alpha=0$ and $0$ for $\abso{\alpha}>0$. In the second we substitute $x=y+\e z$ and obtain
\[
 \e \int_{B_{r/\e + 2C}(a-y)} \int_0^1 (\ud_1 f) (y+t\e z, y)z\, \ud t\, (\pd_{x+y}^\alpha \tilde \phi)(\e, y, y-\e z)\e^n\, \ud z
\]
By the support property of smoothing kernels (Definition \ref{localsmoothkern} (i)) we only have to integrate over a bounded set and the integrand is uniformly bounded on all $x$ and $y$ in question, so this integral is $O(\e)$.
\end{proof}

\begin{corollary}\label{blahfoo}For any $\alpha \in \bN_0^n$ we have
\begin{enumerate}
\item[(i)] For any $f \in C^\infty(\Omega \times \Omega)$ and $K \csub \Omega$ we have
\[ \sup_{y \in K} \abso{ \int f(x,y) (\pd_x^\alpha \tilde \phi)(\e,x,y)\, \ud y - (\pd_y^\alpha f)(x,x)} = O(\e)
\]
 \item[(ii)] For any $f$ as in Proposition \eqref{ableit} we have
\[ \sup_{y \in \Omega} \abso{ \int f(x,y) (\pd_y^\alpha \tilde \phi)(\e,x,y)\, \ud x - (\pd_x^\alpha f)(y,y)} = O(\e) \]
\end{enumerate}
\end{corollary}
\begin{proof}
(i) We perform induction on $\abso{\alpha}$. The case $\alpha=0$ was resp.~mentioned after Definition \ref{localsmoothkern} resp.\ handled in Proposition \eqref{ableit}. For $\abso{\alpha}>0$ we have (by induction or combinatorically) the identity $\pd_x^\alpha = \pd_{x+y}^\alpha - \sum_{0 < \beta \le \alpha}\binom{\alpha}{\beta}\pd_y^\beta \pd_x^{\alpha-\beta}$ (where $\binom{\alpha}{\beta} \coleq \binom{\alpha_1}{\beta_1} \dotsm \binom{\alpha_n}{\beta_n}$) thus using partial integration the integral
\[ \int f(x,y)(\pd_x^\alpha \tilde \phi)(\e,x,y)\,\ud y \]
is given by
\[ O(\e) - \sum_{0 < \beta \le \alpha}\binom{\alpha}{\beta} (-1)^\beta \int (\pd_y^\beta f)(x,y)(\pd_x^{\alpha-\beta}\tilde \phi)(\e,x,y)\, \ud y \]
From the result for $\abso{\alpha}-1$ we see that the integral converges to $(\pd_y^\alpha f)(x,x)$ uniformly and of order $O(\e)$, so (i) follows because $\sum_{0 < \beta \le \alpha}\binom{\alpha}{\beta}(-1)^\beta = -1$. (ii) is done in the same way.
\end{proof}

\section*{Acknowledgements}This research has been supported by START-project Y237 and project P20525 of the Austrian Science Fund and the Doctoral College 'Differential Geometry and Lie Groups' of the University of Vienna.

\end{document}